\documentclass[12pt]{amsart}
\usepackage[lite, alphabetic, nobysame]{amsrefs}
\usepackage{amsmath,mathtools}
\usepackage{amssymb}
\usepackage{srcltx}
\newcommand{\MB}{\boldsymbol{M}}       
\newcommand{\NB}{\boldsymbol{N}}       
\newcommand{\AB}{\boldsymbol{A}}       
\newcommand{\BB}{\boldsymbol{B}}       
\newcommand{\CB}{\boldsymbol{C}}
\newcommand{\DB}{\boldsymbol{D}}
\newcommand{\KB}{\boldsymbol{K}}
\newcommand{\FB}{\boldsymbol{F}}
\newcommand{\LB}{\boldsymbol{L}}

\newcommand{\R}{\mathbb{R}}       
\newcommand{\N}{\mathbb{N}}       
\newcommand{\U}{\mathbb{U}}

\newcommand{\KC}{\mathcal{K}}
\newcommand{\FC}{\mathcal{F}}
\newcommand{\LC}{\mathcal{L}}

\newcommand{\Iso}{\mathrm{Iso}}
\newcommand{\Aut}{\mathrm{Aut}}

\newcommand{\e}{{\varepsilon}}
\newcommand{\stm}{\setminus}
\newcommand{\res}{\mathord{\upharpoonright}}

\newcommand{\Fraisse}{Fra\"iss\'e}

\newtheorem{theorem}{Theorem}[section]

\newtheorem{coro}[theorem]{Corollary}
\newtheorem{propo}[theorem]{Proposition}

\newtheorem*{lemma*}{Lemma}
\newtheorem{lm}[theorem]{Lemma}
\theoremstyle{definition}
\newtheorem{definition}[theorem]{Definition}
\newtheorem*{definition*}{Definition}
\newtheorem*{Ack}{Acknowledgements}
\newtheorem{example}[theorem]{Example}

\theoremstyle{remark}
\newtheorem{rem}[theorem]{\bf Remark}
\newtheorem{question}{Question}

\newcommand{\brem}{\begin{rem}}
\newcommand{\erem}{\end{rem}}

\newenvironment{theorem*}[1][]{\par\medskip\noindent\textbf{Theorem} #1\textbf{.}\itshape}{\medskip}
\newenvironment{theorem**}[1][]{\par\medskip\noindent\textbf{Definition} #1\textbf{.}}{\medskip}
\newenvironment{theorem***}[1][]{\par\medskip\noindent\textbf{Theorem} #1\textbf{.}\itshape}{\medskip}

\begin{document}
\title{Extendability of automorphisms of generic substructures}
\author{Aristotelis Panagiotopoulos}
\address{Department of Mathematics, 1409 W. Green st., University of Illinois, Urbana, IL 61801, USA}
\email{panagio2@illinois.edu}
\thanks{Research partially supported by University of Illinois Research Board Grant and NSF grant DMS-1001623}

\maketitle
\begin{abstract}
We show that if $g$ is a generic (in the sense of Baire category) isometry of a generic subspace of the Urysohn metric space $\U$, then $g$ does not extend to a full isometry of $\U$. The same holds for the Urysohn sphere $\mathbb{S}$. Let $\boldsymbol{M}$ be a \Fraisse{} $\LC$ structure, where $\LC$ is a relational countable language and $\MB$ has no algebraicity. We provide necessary and sufficient conditions for the following to hold: 
``For a generic substructure $\boldsymbol{A}$ of $\boldsymbol{M}$, every automorphism $f\in\mathrm{Aut}(\boldsymbol{A})$ extends to a full automorphism $\tilde{f}\in\mathrm{Aut}(\boldsymbol{M})$.'' From our analysis, a dichotomy arises and some structural results are derived that, in particular, apply to $\omega$-stable \Fraisse{} structures without algebraicity. 
\end{abstract}

\tableofcontents

\pagebreak

\section*{Introduction}

The (separable) \emph{Urysohn metric space} $(\mathbb{U},\rho)$ was introduced in \cite{Urysohn} and it is the unique, up to isometry, Polish metric  space that satisfies the following properties:
\begin{itemize}
\item[$\cdot$] (ultrahomogeneity) for every two finite isometric subspaces  $A,B\subset\mathbb{U}$ and for every isometry $f:A\to B$, $f$ extends to a full isometry $\tilde{f}$ of $\mathbb{U}$;
\item[$\cdot$] (universality) every Polish metric space is isometric to a subspace of $\mathbb{U}$. 
\end{itemize}    
Huhunai\v{s}vili showed in \cite{Huhu} that $\U$ satisfies a strengthening of the ultrahomogeneity property, attained by replacing the adjective ``finite" in $A$ and $B$ above with ``compact". 

There are spaces that enjoy a much stronger version of homogeneity. Consider for example the Euclidean metric space $\mathbb{R}^m$. Then, it is true that for every two, possibly infinite, metric subspaces $A,B$ and every isometry $f:A \to B$, there is an isometry $\tilde{f}$ of the whole space $\R^m$ that extends $f$. In the case of $\U$, it was shown by Melleray in \cite{MellerayInitial} that Huhunai\v{s}vili's result cannot be extended further, i.e., if $X$ is non-compact space then there are isometric copies $A,B$ of $X$ in $\U$ and an isometry $f:A\to B$ that does not extend to an isometry $\tilde{f}$  of $\U$. It is worth noting here that on the other hand, for any separable metric space $X$, adopting Uspenskij's use of  Kat\v{e}tov's tower construction \cite{Uspenskij,Katetov} we can find copies $A,B$ of $X$ in $\U$ so that any isometry $f:A\to B$ extends to a global isometry. 

The approach that we consider here is of a slightly different sort. For every Polish metric space $(X,d)$ there is canonical Polish topology for the hyperspace $\FC(X)$ of all closed subsets of $X$, namely, the Wijsman topology \cite{Wijsman,Beer}. We say that for a \emph{generic subspace} of $X$ a certain property holds if the set of all closed subsets of $X$ that have this property is a comeager subset of $\FC(X)$ in the Wijsman topology. The question which motivates our investigation is whether for a generic subspace $F$ of  $\U$, every self-isometry of $F$ extends to an isometry of $\U$.  It turns out as a consequence of Lemma \ref{TheoremUrysohnAbsorbsPoints} that for a generic subspace $F$ of $\U$ the space $F$ is isometric to $\U$ itself. Since the generic subspace is of one isometry type, it therefore makes sense to ask whether for a \emph{generic pair} $F_1, F_2$ of subspaces of $\U$ every isometry $f:F_1\to F_2$ extends to a global isometry. Here, we identify pairs of closed subsets of $\U$ with points in $\FC(\U)\times\FC(\U)$. Keeping in mind that $\U$ is just an instance of the general problem that we are going to deal with, consider the following definitions.

For every two isometric Polish metric spaces $X,Y,$ we write $\Iso(X)$ for the space of all isometries of $X$, and $\Iso(X,Y)$ for the space of all isometries from $X$ onto $Y$. The spaces $\Iso(X)$ and $\Iso(X,Y)$ are Polish, equipped with the pointwise convergence topology (see, for example, \cite{bible}*{Section 9B}). If $A,B$ are isometric subsets of $X$ we write $\mathcal{E}(A)$ to denote the set of all self-isometries of $A$ that extend to a global isometry of $X$ and similarly by $\mathcal{E}(A,B)$ we denote the set of all isometries from $A$ to $B$ that extend to a global isometry of $X$.
\begin{definition*}
Let $A$ be a subspace of $\U$. We say that $A$ is a \emph{global subspace} if $\mathcal{E}(A)=\Iso(A)$, a \emph{non-global subspace} if $\mathcal{E}(A)\subsetneq\Iso(A)$, or a \emph{strongly non-global subspace} of $\U$ if $\mathcal{E}(A)$ is a meager subset of $\Iso(A)$.

Similarly, we say that a pair $A,B$ of isometric subspaces of $\U$ is a \emph{global pair} if $\mathcal{E}(A,B)=\Iso(A,B)$, a \emph{non-global pair} if $\mathcal{E}(A,B)\subsetneq\Iso(A,B)$, or a \emph{strongly non-global pair} in $\U$ if $\mathcal{E}(A,B)$ is a meager subset of $\Iso(A,B)$. 
\end{definition*}
In Chapter \ref{ChapterUrysohn}, we give a ``strongly negative" answer to both of our initial questions. The same results also follow for the Urysohn sphere $\mathbb{S}$ by appropriating our methods in the bounded metric context.
\begin{theorem***}[\textbf{\ref{UrysohnMainTheorem}}]
Let $\U$ be the Urysohn space. Then, the generic subspace $F\in\mathcal{F}(\U)$ as well as the generic pair $A,B,$ of subspaces of $\mathcal{F}(\U)$ are strongly non global. 
\end{theorem***}

Before we develop the theory for the Urysohn space, we undertake the task of answering the same questions in the countable setting. Where instead of a metric space, we work with ultrahomogeneous countable $\LC$-structures $\MB$ of some relational language $\LC$. A structure $\MB$ is called \emph{ultrahomogeneous} if every isomorphism between finite substructures of $\MB$ can be extended to a full automorphism of $\MB$. The rationals with their natural ordering $(\mathbb{Q},\leq)$ and the random graph $(\mathbb{G},R)$ are two classical examples of countable ultrahomogeneous structures. Working in this new context we can ask the same questions if we first make the natural changes: we replace the word ``subspace" with the word ``substructure," the word ``isometry" with the word ``isomorphism," and we identify the space of all substructures of $\MB$ with the Cantor space $2^M$. 

 If $\MB$ is an $\LC$-structure, $\mathrm{Age}(\MB)$ denotes the class of all  finite $\LC$-structures that can be embedded in $\MB$. Countable ultrahomogeneous structures are also called \emph{\Fraisse{} structures} because each such structure $\MB$ can be attained as a limit (the so called \emph{Fra\"{i}ss\'e limit}) over $\mathrm{Age}(\MB)$. The class $\mathrm{Age}(\MB)$ is called a \emph{\Fraisse{} class} if $\MB$ is ultrahomogeneous. This approach, introduced by \Fraisse{} in \cite{Fraisse}, allowed the systematic study of infinite ultrahomogeneous structures $\MB$ through the study of the combinatorial properties of the finite objects lying in  $\mathrm{Age}(\MB)$.

 In this paper, we will limit our study to structures which have no algebraicity (see \cite{Hodges}, \cite{oligomorphic}, or Chapter \ref{ChapterSAP} for a definition). One of the known consequences that we also derived here from Lemma \ref{pair}, is that if $\MB$ has no algebraicity then for a generic substructure $\AB$ of $\MB$ the structure $\AB$ is isomorphic to $\MB$. In Chapter \ref{ChapterSAP} we will see that for a \Fraisse{} structure $\MB$ without algebraicity the generic substructure of $\MB$, as well as the generic pair of substructures of $\MB$, is either global or strongly non-global. Moreover, we will reflect the dividing line of the aforementioned dichotomy to the following, central in this paper, property of \Fraisse{} classes.

\begin{theorem**}[\textbf{\ref{splits}}]
Let $\mathcal{K}$ be a \Fraisse{} class and let $\boldsymbol{C}\in\mathcal{K}$. We say that $\boldsymbol{C}$ \emph{splits} $\mathcal{K}$ if for every $\boldsymbol{D}\in\KC$ and for every embedding $i:\CB\to\DB$ , there are structures $\boldsymbol{D}_1,\boldsymbol{D}_2\in\KC$, embeddings $j_1:\DB\to\DB_1$ and $j_2:\DB\to\DB_2$  and a bijection $f:D_1\to D_2,$ such that:
\begin{itemize}
\item[$\cdot$] $f\circ{}j_1=j_2$;
\item[$\cdot$] $f\res_{D_1\setminus C}$ is an isomorphism between $\langle D_1\setminus C\rangle_{\boldsymbol{D}_1}$ and $\langle D_2\setminus C\rangle_{\boldsymbol{D}_2}$;
\item[$\cdot$] $f$ is not an isomorphism between $\DB_1$ and $\DB_2$.
\end{itemize}  
\end{theorem**}
We say that $\mathcal{K}$ \emph{splits} if there is a $\boldsymbol{C}\in\mathcal{K}$ that splits $\mathcal{K}$. In the language of graphs, a typical example of a \Fraisse{} class $\KC$ that splits is the age of the random graph and a typical example of a \Fraisse{} class that does not split is the age of the countable complete graph. The main result of Chapter \ref{ChapterSAP} will be the following theorem.
 
\begin{theorem*}[\textbf{\ref{theorem}}]
Let $\boldsymbol{M}$ be a \Fraisse{} structure that has no algebraicity and let $\mathcal{K}$ be the corresponding \Fraisse{} class. 
\begin{enumerate}
\item If $\mathcal{K}$ splits then the generic substructure $\AB$ of $\MB$ is a strongly non-global substructure and the generic pair $\AB,\BB$ in $\MB$ is a strongly non-global pair.
\item If $\mathcal{K}$ does not split then  the generic substructure $\AB$ of $\MB$  is a global substructure and the generic pair $\AB,\BB$ in $\MB$  is a global pair.
\end{enumerate} 
\end{theorem*} 
Structures $\MB$ with corresponding age $\KC$ that does not split seem to be simpler than the ones having age that splits. In Chapter \ref{ChapterStructuralConsequences}, we present some structural consequences for the structures $\MB$ that have age which does not split. Theorem \ref{theorem omega stable--> does not split } states that $\omega$-stable \Fraisse{} limits with no algebraicity have ages that do not split. We also provide an example showing that the converse is not true. Theorem \ref{theorem_Wreath_Product} is a structural result regarding automorphism groups of \Fraisse{} limits which have no algebraicity and an age that does not split.      

The proofs of the main theorems of Chapters \ref{ChapterSAP} and \ref{ChapterUrysohn} use infinite games. In Chapter \ref{Chapter Game}, we define the Banach Mazur game and we state the main result regarding this game that we are going to use later in the paper. A short note on the Wijsman hyperspace topology is  given in Chapter \ref{ChapterWijsman}. 

 \begin{Ack}
 I want to thank S\l{}awek Solecki for bringing to my attention the main question undertaken in this paper as well as for his help and guidance throughout the research. I would also like to thank the anonymous referee for many helpful comments and suggestions which gave to this paper its final shape.
 \end{Ack}

\section{The Banach Mazur game $G^{**}(E,X)$}\label{Chapter Game}
Let $X$ be a topological space and let $A$ be a subset of $X$. The set $A$ is called meager if it is a countable union of nowhere dense in $X$ sets. The collection of all meager sets of $X$ forms a $\sigma$-ideal. Therefore, meager sets can be thought of as topologically small sets and their complements, the so called comeager sets, can be thought of as topologically large sets. Equivalently, we can directly define comeager sets as exactly those subsets of $X$ which contain a dense in $X$, $G_\delta$ subset of $X$.  A useful  technique used to prove that a subset $A$ of a topological space $X$ is comeager involves an infinite game known as Banach Mazur game $G^{**}(A,X)$. Here, we are going to review in short the Banach Mazur game. For a more detailed exposure on the notion of meager and comeager sets, Banach Mazur games,  as well as the proof of the main theorem of this chapter see \cite{bible}.

Let $X$ be a topological space and let $A$ be a subset of $X$. The \emph{Banach Mazur game} $G^{**}(A,X)$ is a game played with 2 players, player I and player II. Player I starts by choosing an open subset $U_0$ of $X$ and then player II replies with an open subset $V_0$ of $U_0$. Then, player I plays further a new open set $U_1$ with $U_1\subset V_0$ and so on. The game continues this way with the two players alternating turns and together defining a decreasing sequence of open sets. A run of the game looks as follows:
\[U_0\supset V_0 \supset  U_1 \supset V_1\supset\ldots\supset U_m\supset V_m\supset\ldots,\]
and player II wins this run of the game if and only if $\bigcap_n V_n(=\bigcap U_n)\subset A$. A winning strategy for player II is roughly a preestablished rule that tells player II which open set $V_n$ to reply given an initial segment $(U_0,V_0,\ldots,U_n)$ of any possible run of the game and that moreover, this rule leads always to victory for Player II. The following theorem is the main result that we are going to use regarding the $G^{**}(A,X)$ game.

\begin{theorem}[Banach-Mazur, Oxtoby]\label{banach-mazur}
Let $X$ be a nonempty topological space. Then $A$ is comeager if and only if player II has a winning strategy in $G^{**}(A,X)$.
\end{theorem}

\section{Countable \Fraisse{} structures without algebraicity}\label{ChapterSAP}

The main result of this section is Theorem \ref{theorem}. In what follows,  $\mathcal{L}$ will always be a countable, relational language and  $\boldsymbol{M}$ will always  be a countable $\mathcal{L}$-structure. We write $\mathrm{Age}(\boldsymbol{M})$ for the class of all finite $\mathcal{L}$-structures that can be embedded in $\boldsymbol{M}$. We will use lightface letters for the subsets of the domain of $\boldsymbol{M}$ and boldface letters for the induced substructures. For example, $M$ denotes the domain of $\boldsymbol{M}$ and for every $A\subset M$, we write $\boldsymbol{A}=\langle A \rangle_{\MB}$ for the substructure of $\boldsymbol{M}$ generated by $A$. We will also use the notation $\AB^c$ for the structure $\langle A^c\rangle_{\MB}$.  Notice that due to the fact that $\LC$ is relational, we have a bijective correspondence between subsets of $M$ and substructures of $\MB$.   Without any loss of generality we assume from now on that the domain of $\MB$ is the set of natural numbers $\N$. We now have a natural bijective correspondence between substructures $\AB$ of $\MB$ and points in the Cantor space $\mathcal{C}=2^\mathbb{N}$, given by the characteristic function $\chi_A$ of $A$.

The Polish group $S_\infty$ is the group of all bijections of the domain of $\MB$ endowed with the pointwise convergence topology.  We denote with $\mathrm{Aut}(\boldsymbol{M}),$ the group of all automorphisms of the structure $\boldsymbol{M}$. The group $\Aut(\MB)$ is a closed subgroup of $S_\infty$ and therefore a Polish group inheriting the topology from $S_\infty$ . If $\AB,\BB$ are substructures of $\MB$, we write $\Iso(\AB,\BB)$  to denote the space of all isomorphisms from $\AB$ to $\BB$. Again,  endowed with the pointwise convergence topology, $\Iso(\AB,\BB)$ is a Polish space.

An $\mathcal{L}$-structure $\boldsymbol{M}$ is called \emph{ultrahomogeneous} if every isomorphism between finite substructures $\boldsymbol{A},\boldsymbol{B}$ of $\boldsymbol{M}$ extends to a full automorphism of $\boldsymbol{M}$. Countable ultrahomogeneous structures are also known as \emph{\Fraisse{} structures} or \emph{\Fraisse{} limits}. We will further assume here that \Fraisse{} structures are always of non-finite cardinality. We will review some basic facts regarding \Fraisse{} structures. For a more detailed exposition someone may want to consult \cite{Hodges}. If $\boldsymbol{M}$ is a \Fraisse{} structure, and $\mathcal{K}=\mathrm{Age}(\boldsymbol{M})$, then $\mathcal{K}$ has the following properties:
\begin{itemize}
\item[(i)] Hereditary Property(HP): if $\boldsymbol{A}\in\mathcal{K},$ and $\boldsymbol{B}$ is a substructure of $\boldsymbol{A}$, then $\boldsymbol{B}\in\mathcal{K}$; 
\item[(ii)] Joint Embedding Property(JEP): if $\boldsymbol{A},\boldsymbol{B}\in\mathcal{K}$, there is $\boldsymbol{C}$ in $\mathcal{K}$ such that both $\boldsymbol{A}$ and $\boldsymbol{B}$ embed in $\boldsymbol{C}$;
\item[(iii)] Amalgamation Property (AP): if $\boldsymbol{A}, \boldsymbol{B}, \boldsymbol{C}$ in $\mathcal{K}$ and $f:\boldsymbol{A}\to \boldsymbol{B}$, $g:\boldsymbol{A}\to \boldsymbol{C}$ embeddings, there is $\boldsymbol{D}\in\mathcal{K}$ and embeddings $i:\boldsymbol{B}\to \boldsymbol{D}$, $j:\boldsymbol{C}\to \boldsymbol{D}$, such that $if=jg$;  
\item[(iv)] every subclass of pairwise non-isomorphic structures of $\mathcal{K}$ is at most countable, and
\item[(v)] $\mathcal{K}$ contains structures of arbitrary large, finite size. 
\end{itemize}  
If $\mathcal{K}$ is a class of finite $\mathcal{L}$-structures and has the properties (i)-(v), we say that $\mathcal{K}$ is a \emph{\Fraisse{} class}. \Fraisse{}'s theorem \cite{Fraisse} establishes the converse direction: if $\mathcal{K}$ is a \Fraisse{} class, then there is an $\mathcal{L}$-structure $\boldsymbol{M}=\boldsymbol{M}(\mathcal{K})$, unique up to isomorphism, such that $\boldsymbol{M}$ is countably infinite, ultrahomogeneous  and $\mathcal{K}=\mathrm{Age}(\boldsymbol{M})$. 

We say that an $\LC$-structure $\MB$ has no \emph{algebraicity} if the pointwise stabilizer in $\Aut(\MB)$ of an arbitrary tuple of $\MB$ has no finite orbits in its natural action on $\MB$. Here we are going to work only with \Fraisse{} structures $\MB$ that have no algebraicity. A \Fraisse{} structure $\MB$ has no algebraicity if and only if the associated \Fraisse{} class $\mathcal{K}$ satisfies the strong amalgamation property, defined as follows: 
\begin{itemize}
\item[(SAP)] we say that $\mathcal{K}$ has the \emph{strong amalgamation property} if for every
$\boldsymbol{A}, \boldsymbol{B}, \boldsymbol{C}\in\mathcal{K}$ and $f:\boldsymbol{A}\to \boldsymbol{B}$, $g:\boldsymbol{A}\to \boldsymbol{C}$ embeddings, there is $\boldsymbol{D}\in\mathcal{K}$ and embeddings $i:\boldsymbol{B}\to \boldsymbol{D}$, $j:\boldsymbol{C}\to \boldsymbol{D}$, such that $if=jg$ and
\[i(B)\cap j(C)=if(A)=jg(A).\]
\end{itemize}
For the interested reader, a proof of this fact can be found in \cite{oligomorphic}.
\begin{example} \label{example}
The list of countable ultrahomogeneous structures with no algebraicity includes the following examples.
\begin{itemize}
\item[$\cdot$] $\boldsymbol{M}_1=(\N)$, the empty-language, countable structure.
\item[$\cdot$] $\boldsymbol{M}_2=(\sqcup_{i\in\N}G_i,R)$, the disjoint union of countably many countable complete graphs ($\MB_2 \models R(a,b)$ if and only if $a, b\in G_i$ for some $i$).
\item[$\cdot$] $\boldsymbol{M}_3=(\mathbb{G},R)$, the random graph. 
\item[$\cdot$] $\boldsymbol{M}_4=(\mathbb{Q},\leq)$, the countable dense linear order without endpoints.
\item[$\cdot$] $\boldsymbol{M}_5=(Q\mathbb{U},\{d_i\}_{i\in\mathbb{Q}^+})$, the rational Urysohn metric space. 
\end{itemize} 
\end{example}
Returning to the main question that concerns us here, notice that if we pick $\MB$ to be any structure among $\boldsymbol{M}_2,\boldsymbol{M}_3,\boldsymbol{M}_4,\MB_5$, and $\NB$ some infinite substructure of $\MB$, there are always ways of embedding $\boldsymbol{N}$ in $\boldsymbol{M}$ so that every  automorphism of $\boldsymbol{N}$ extends to a global automorphism of 
$\boldsymbol{M}$, and ways of embedding $\boldsymbol{N}$ in $\boldsymbol{M}$, so that not every automorphism of $\boldsymbol{N}$ extends to a global automorphism of $\boldsymbol{M}$. Consider for example the random graph $\MB_3=(\mathbb{G},R)$ and take $\boldsymbol{N}$ to be the  structure that remains if we remove from $\mathbb{G}$  one point $x$. Using a back and forth system we can create an automorphism $f$ of $\boldsymbol{N}$ that sends all points connected to $x$, to the points not connected to $x$ and vice versa. Of course, $f$ cannot be extended to an automorphism of $\boldsymbol{M}_3$. On the other hand, using the Kat\v{e}tov tower construction for graphs, as done in \cite{Tower}, we can embed any countable graph $\NB$ in $\MB_3$ in such a way that every automorphism of $\NB$ extends to an automorphism of $\MB_3$. 
\begin{definition}
Let $\AB,\BB$ be two isomorphic substructures of $\MB$. We write $\mathcal{E}(\AB)$ to denote the set of all self-isomorphisms of $\AB$ that extend to a global isomorphism of $\MB$ and by $\mathcal{E}(\AB,\BB)$ we denote the set of all isomorphisms from $\AB$ to $\BB$ that extend to a global isomorphism of $\MB$.

We say that $\AB$ is a \emph{global substructure} if $\mathcal{E}(\AB)=\Iso(\AB)$, a \emph{non-global substructure} if $\mathcal{E}(\AB)\subsetneq\Iso(\AB)$, or a \emph{strongly non-global substructure} if $\mathcal{E}(\AB)$ is a meager subset of $\Iso(\AB)$. Similarly, we say that the pair $\AB,\BB$ is a \emph{global pair} if $\mathcal{E}(\AB,\BB)=\Iso(\AB,\BB)$, a \emph{non-global pair} if $\mathcal{E}(\AB,\BB)\subsetneq\Iso(\AB,\BB)$, or a \emph{strongly non-global pair}  if $\mathcal{E}(\AB,\BB)$ is a meager subset of $\Iso(\AB,\BB)$. 
\end{definition}

 Recall now that we have identified with $2^\N$ the space of all substructures of $\MB$. We say that for a \emph{generic substructure} $\AB$ of $\MB$ a certain property holds if and only if the set  of all substructures $\AB$ of $\MB$ that have this property is a comeager subset of $2^\N$. Similarly, we say that for a \emph{generic pair of substructures of} $\MB$, a certain property holds if and only if the set of pairs $\AB,\BB$ of substructures of $\MB$ that have this property form a comeager subset of $2^\N\times2^\N$. In what follows, we are going to see that for a \Fraisse{} limit without algebraicity, the generic substructure of $\MB$, as well as the generic pair of substructures of $\MB$,  is either global or strongly non global.  We are also going to reflect this dichotomy to the satisfiability or non-satisfiability of a certain property of the \Fraisse{} class $\KC$ corresponding to $\MB$. We begin by giving an example to offer some intuition regarding the forthcoming Definition \ref{splits}.
\begin{example}
Let $\KC_3$ be the \Fraisse{} class of all finite graphs. Let $\DB$ be any finite graph and let $c\in D$. Let also $D_1=D\cup\{w\}$, where $w\not\in D$, and consider any graph $\DB_1\in\KC_3$ with domain $D_1$ such that $\DB_1\res D= \DB$. Notice that whatever $\DB_1$ is chosen to be, we can find another graph $\DB_2\in \KC_3$ on the same domain $D_2=D_1$, such that:
\begin{itemize}
\item[$\cdot$] $\DB_2\res D=\DB$;
\item[$\cdot$] $\DB_2\res \big(D_2\stm\{c\}\big)=\DB_1\res \big(D_1\stm\{c\}\big)$, and
\item[$\cdot$] $\DB_2\models R(c,w)$ if and only if $\DB_1\models\neg R(c,w)$.
\end{itemize} 
This basically says that $\DB_1$ and $\DB_2$ are not isomorphic, but the only way to witness this fact is by checking the relations between $c$ and $w$. Notice moreover, that the same is not true for the \Fraisse{} class $\KC_2$ corresponding to $\MB_2$ from Example \ref{example} above: if $\DB$ is any graph from $\KC_2$ with $c,u\in D$ connected and $\DB '\in\KC_2$ is extending $\DB$, then relationship between $c$ and any new point $w$ of $\DB '$ is uniquely determined by the relation between $u$ and $w$.
\end{example}
\begin{definition}\label{splits}
Let $\mathcal{K}$ be a Fra\"iss\'e class, and let $\boldsymbol{C}\in\mathcal{K}$. We say that $\boldsymbol{C}$ \emph{splits} $\mathcal{K}$ if for every $\boldsymbol{D}\in\KC$ and for every embedding $i:\CB\to\DB$ , there are structures $\boldsymbol{D}_1,\boldsymbol{D}_2\in\KC$, embeddings $j_1:\DB\to\DB_1$ and $j_2:\DB\to\DB_2$  and a bijection $f:D_1\to D_2,$ such that:
\begin{itemize}
\item[$\cdot$] $f\circ{}j_1=j_2$;
\item[$\cdot$] $f\res_{D_1\setminus C}$ is an isomorphism between $\langle D_1\setminus C\rangle_{\DB_1}$ and $\langle D_2\setminus C\rangle_{\DB_2}$;
\item[$\cdot$] $f$ is not an isomorphism between $\DB_1$ and $\DB_2$.
\end{itemize}  
Where, in order to keep the notation simple we write $D_1\setminus C$ instead of $D_1\setminus j_1(i(C))$, etc.
We say that $\mathcal{K}$ \emph{splits} if there is a $\boldsymbol{C}\in\mathcal{K}$ that splits $\mathcal{K}$.
\end{definition}
\begin{rem}
If $\KC_1,\ldots,\KC_5$ are the \Fraisse{} classes that correspond to the structures $\MB_1,\ldots,\MB_5$ of the Example \ref{example}, then $\KC_3,\KC_4$ and $\KC_5$ split and $\KC_1,\KC_2$ do not split.   
\end{rem}
Let $\boldsymbol{M}=\boldsymbol{M}(\mathcal{K})$ be the \Fraisse{} structure associated to $\mathcal{K}$ and let $X$ be a finite subset of $M$. We denote by $\mathcal{L}_X,$ the language obtained by adding to $\mathcal{L}$ a constant $c_x$ for every $x\in X$.
\begin{definition}
Let $\MB$ be a \Fraisse{} structure and let $\boldsymbol{X}$ be a finite substructure of $\MB$. By a \emph{realized quantifier free type}(rqf-type) $p=p(y)$ over $\boldsymbol{X}$ we mean a set of quantifier free $\mathcal{L}_X$-formulas $\phi$ in one variable $y$, for which there is a $z\in M$ such that
$$\phi\in p \Leftrightarrow \boldsymbol{M}\models\phi(z).$$
As a slight abuse of notation, we will not exclude the possibility of $X$ being the empty set in this definition.
Finally, we say that the rqf-type $p$ over $\boldsymbol{X}$ is \emph{non-trivial} if for all $x\in X$, the formula $\phi(y)\equiv (y=c_x)$ does not belong to $p$.    
\end{definition}
 Notice that if $\MB$ is a \Fraisse{} structure then $p$ is a rqf-type over a finite substructure $\boldsymbol{X}$ of $\boldsymbol{M}$ if and only if  $\langle X,z\rangle_{M}\in\mathcal{K}$. So, it makes sense to talk about realized quantifier free types over a finite structure $\boldsymbol{X}$, whenever $\boldsymbol{X}\in\mathcal{K}$.
\begin{definition}
Let $\boldsymbol{X},\boldsymbol{X}'\in\mathcal{K}$ be $\mathcal{L}$-structures of the same  size, and let $f:X\to X'$ be a bijection between their domains. Let also $p=p(y)$ be a realized quantifier free type over $\boldsymbol{X}$. We define $f[p]=f[p](y)$ to be following set of $\mathcal{L}_{X'}$-formulas:
$$\varphi(y,c_{x_1},\ldots,c_{x_n})\in f[p]\,\,\Leftrightarrow\,\,\varphi(y,c_{f^{-1}(x_1)},\ldots,c_{f^{-1}(x_n)})\in p.$$  
\end{definition}

Notice that if $p$ is a rqf-type and \,$f$\, is an isomorphism between $\boldsymbol{X}$ and $\boldsymbol{X}'$, then the set of quantifier free formulas $f[p]$ is a rqf-type over $\boldsymbol{X}'$.

\begin{definition}
Let $\AB$ be a substructure of  a \Fraisse{} structure $\MB$. We say that $\AB$ \emph{absorbs points} if for every finite subset $X$ of $M$ and for every non-trivial rqf-type $p$ over $\boldsymbol{X}$, there is an $a\in A$ such that $\boldsymbol{M}\models p(a)$.
\end{definition}
Notice that if $\AB$ absorbs points in $\boldsymbol{M}$ then $A$ is not empty. In particular, $A$ is infinite and $\boldsymbol{A}$ is isomorphic to $\boldsymbol{M}$.

For the general \Fraisse{} structure we cannot hope that we can find even one subset $A$ of $M$ such that both $\AB$ and $\AB^c$ absorb points.  For example, take any \Fraisse{} structure $\boldsymbol{M}$. Extend the language $\mathcal{L}$ to $\mathcal{L}'=\mathcal{L}\cup\{u\}$ so that it includes a new unary predicate $u$ and turn $\boldsymbol{M}$ into a $\mathcal{L}'$ structure $\boldsymbol{M}'$ by letting for some $x_0\in M$ the following:  
\[\boldsymbol{M}'\models u(x_0)\quad\text{and}\quad\boldsymbol{M}'\models \forall x\big((x\neq x_0) \Rightarrow \neg u(x)\big).\]
This new structure is a \Fraisse{} limit of a new class $\mathcal{K}',$ but for no subset $A$ of $M'$ both $\AB$ and $\AB^c$ absorb points since then both $A$ and $A^c$ should contain a point that satisfies $u$. However, if we assume that our \Fraisse{} structure has no algebraicity or equivalently if the corresponding \Fraisse{} class has $\mathrm{SAP}$, then we get the following result. 
\begin{lm}\label{pair}
Let $\boldsymbol{M}$ be the \Fraisse{} limit of the \Fraisse{} class $\mathcal{K}$. Assume moreover, that $\mathcal{K}$ has  $\mathrm{SAP}$. Then, for a generic substructure $\AB$ we have that both $\AB$ and $\AB^c$ absorb points.    
\end{lm}
\begin{proof}
Recall that we identify the domain $M$ with the set of natural numbers. We will show that the set $\mathcal{A}$, of all the subsets $A$ of $M$ for which both $\AB$ and $\AB^c$ absorb points is a dense $G_\delta$ subset of $2^M$. Let
$$\mathcal{I}=\{ \,(X,p)\,:\,X\subset M,\,\text{finite,}\,\,\, p\,\,\, \text{a non trivial rqf-type over}\,\,\boldsymbol{X}  \},$$
and notice that $\mathcal{I}$ is countable.  Let $\{i_m : m\in\N\}$ be an enumeration of $\mathcal{I}$. For fixed $i=(X,p),$ let $N_i$ be the subset of $M$, of  all elements $n$ such that $\boldsymbol{M}\models p(n)$. We have that
\[\mathcal{A}=\bigcap_{i\in\mathcal{I}}  \,\,\bigcup_{\substack{n\neq m \\n,m\in N_i}} \big\{ x\in 2^M : x(n)=1, x(m)=0\big\}.\]
Therefore $\mathcal{A}$ is a $G_\delta$ subset of $2^M$. To see that $\mathcal{A}$ is also dense in $2^M$, notice that since $\MB$ has no algebraicity, for every finite substructure $\boldsymbol{X}$ of $\MB$ and every rqf-type $p$ over $\boldsymbol{X}$ there are infinitely many points $a\in M$ with $\MB\models p(a)$.
\end{proof}
\begin{coro}\label{corollaryGenericIsIsomo}
Let $\MB$ be a \Fraisse{} structure which has no algebraicity. Then, for a generic substructure $\AB$ of $\MB$, the structure $\AB$ is isomorphic to $\MB$.
\end{coro}
\begin{proof}
Use the fact that $\AB$ absorbs points to built a back and forth system between $\AB$ and $\MB$.
\end{proof}
It is also immediate from  Lemma \ref{splits} above and the fact that Cartesian product of comeager sets is comeager, that under the  assumptions of Lemma \ref{splits}, for a generic pair $\AB,\BB,$ of substructures of $\MB$, all $\AB,\AB^c,\BB,\BB^c,$ absorb points.

A \emph{partial isomorphism} of $\MB$ is a map $f:N\to M$ with $N\subset M$ which happens to be an isomorphism between $\NB$ and $\langle f(N)\rangle_{\MB}$. We will write $\mathrm{dom}f$ to denote the domain of $f$. We say that $f$ is a \emph{finite partial isomorphism} if $f$ is a partial isomorphism with finite domain. If $f_1,f_2$ are two partial isomorphisms of $\MB$, we say that $f_1$ and $f_2$ are \emph{compatible} if there is a partial  isomorphism $f$ of $\MB$ that extends both $f_1$ and $f_2$.

\begin{lm}\label{not extend}
Let $\mathcal{K}$ be a \Fraisse{} class that splits and has the $\mathrm{SAP}$ and let $\boldsymbol{M}=\boldsymbol{M}(\mathcal{K})$ be the \Fraisse{} limit of $\mathcal{K}$. Let also  $A,B\subset M$ such that all $\AB,\AB^c,$ and $\BB,\BB^c$, absorb points. Then, $\mathcal{E}(\AB,\BB)$ is a meager subset of $\mathrm{Iso}(\boldsymbol{A},\boldsymbol{B})$.
\end{lm}

\begin{proof}
Let $A,A^c,$ and $B,B^c$, as above and let $\mathcal{N}=\mathrm{Iso}(\boldsymbol{A},\boldsymbol{B})\setminus\mathcal{E}(\AB,\BB).$
We will show that there is a winning strategy for player II in the Banach-Mazur game $G^{**}\big(\mathcal{N},\mathrm{Iso}(\boldsymbol{A},\boldsymbol{B})\big)$. Therefore,  $\mathcal{N}$ is comeager subset of $\mathrm{Iso}(\boldsymbol{A},\boldsymbol{B})$ by Theorem \ref{banach-mazur}. 

Take $\boldsymbol{C}\in\mathcal{K}$, such that $\boldsymbol{C}$ splits $\mathcal{K}$. Since  $\AB^c$ absorbs points, we can realize $\boldsymbol{C}$ inside $\boldsymbol{A}^c$. Let $\boldsymbol{C}_*$ be any such a realization. Let also $\{g_k:k\in\N\}$ be an enumeration of all finite partial isomorphisms of $\MB$ who are compatible with some $h\in\mathcal{E}(\AB,\BB)$ and whose domain includes $C_*$. Obviously, every $h\in\mathcal{E}(\AB,\BB)$ is compatible with some $g_k$ in the above list. Player II will pick his moves so that no matter what player I does, the resulting map $h$ of the play will belong in $\mathrm{Iso}(\boldsymbol{A},\boldsymbol{B})$ and moreover, $h$ will not be compatible with any $g_k$. Therefore, by the above observation, $h$ will belong in $\mathcal{N}$.

 For the first task, notice that by incorporating additionally in the moves of player II a ``back and forth" system between $\AB$ and $\BB$, we can assume without the loss of generality that the result of the play will indeed be an isomorphism $h$ from $\AB$ to $\BB$. Assume now that the game is in its $n$-th step, with $n\geq 0$, and player I has played an open set $U_n\subset\mathrm{Iso}(\boldsymbol{A},\boldsymbol{B})$  which is identified with an partial isomorphism $h_n$, between finite substructures of $\boldsymbol{A}$ and $\boldsymbol{B}$. Player II will proceed as follows: let $k_n$ be the minimum index so that $g_{k_n}$ is compatible with $h_n$ and let $g$ be any finite partial isomorphism compatible with some $h\in\mathcal{E}(\AB,\BB)$ so that $g$ extends both $g_{k_n}$ and $h_n$.  Let $D$ be the domain of $g$ and notice that $C_*\subset D$.  By Definition \ref{splits}, and because $\boldsymbol{C}$ splits $\mathcal{K}$, there are structures $\boldsymbol{D}_1,\boldsymbol{D}_2\in\KC$, embeddings $j_1:\DB\to\DB_1$ and $j_2:\DB\to\DB_2$  and a bijection $f:D_1\to D_2,$ such that:
\begin{itemize}
\item[$\cdot$] $f\circ{}j_1=j_2$;
\item[$\cdot$] $f\res_{D_1\setminus C}$ is an isomorphism between $\langle D_1\setminus C\rangle_{\DB_1}$ and $\langle D_2\setminus C\rangle_{\DB_2}$;
\item[$\cdot$] $f$ is not an isomorphism between $\DB_1$ and $\DB_2$.
\end{itemize}  
Since $A$ absorbs points, we can extend $D$ to $\tilde{D}_1\subset M$ so that $\tilde{\DB}_1\simeq\DB_1$ and all points of $\tilde{D}_1\setminus D$ lie inside $A$. Similarly, since $\langle g(D)\rangle_{\MB}\simeq\DB$ and since $B$ absorbs points, we can extend $g(D)$ to $\tilde{D}_2\subset M$ so that $\tilde{\DB}_2\simeq\DB_2$ and all points of $\tilde{D}_2\setminus g(D)$ lie inside $B$. The function $f$ can be now realized as a bijection $\tilde{f}:\tilde{D}_1\to\tilde{D}_2$ which extends $g$. The function $\tilde{f}$ is not a partial isomorphism of $\MB$ however, if $E$ is any subset  of the domain of $\tilde{f}$ that excludes $C_*$, $\tilde{f}\res_E$ is a partial isomorphism of $\MB$.  

Player II will now reply in his $n$-th round with the open set $V_n$, given by the partial isomorphism $\tilde{h}_n=\tilde{f}\res_{\mathrm{dom}\tilde{f}\cap A}$. Notice that any extension of $\tilde{h}_n$ to an $h\in\mathcal{E}(\AB,\BB)$ is not compatible with $g_{k_n}$. Hence, the game will end with an isomorpism $h=\cup h_n=\cup \tilde{h}_n$ between $\AB$ and $\BB$, which cannot be further extended to include $C_*$ in its domain and therefore $h\in\mathcal{N}$. 
\end{proof}
Together with the Lemma \ref{pair}, Lemma \ref{not extend} proves the one direction of Theorem \ref{theorem}. For the other direction we need first some lemmas.

 Let $\MB$ be the \Fraisse{} limit of a class $\KC$ that does not split. Then, for every $c\in M$ there is a finite $D\subset M$ with $c\in D$ such that for every finite $D_1,D_2\supset D$ and every bijection $f:D_1\to D_2$ we have that : if $f\res_{D_1/\{c\}}$ is an isomorphism between $\langle D_1/\{c\}\rangle_{\MB}$ and $\langle D_2/\{c\}\rangle_{\MB}$, then $f$ is an isomorphism between $\DB_1$ and $\DB_2$. 
 
 In other words, for every $c\in M$, there is a finite $K\subset M$ ($K=D/\{c\}$ above) so that the rqf-type of $c$ over $\KB$ completely determines the rqf-type of $c$ over any finite extension $\FB$ of $\KB$. It will be convenient to settle on the following definition.

\begin{definition}\label{DefDoesNotSplit}
Let $\MB$ be a \Fraisse{} structure, $c\in M$ and $K\subset M$ finite with $c\not\in K$. We say that $\KB$ \emph{controls} $c$ if for every $F_1,F_2\subset M$ finite with $K\subset F_1,F_2$ and every bijection $f: F_1\cup\{c\}\to F_2\cup\{c\}$ with $f\res_{K\cup\{c\}}=\mathrm{id}$ we have that: if $f\res_{F_1}$ is an isomorphism between $\FB_1$ and $\FB_2$ then $f$ is an isomorphism between $\langle F_1\cup\{c\}\rangle_{\MB}$ and $\langle F_2\cup\{c\}\rangle_{\MB}$.
\end{definition}

In the following lemma we record some trivial facts regarding the above notion. 
 
\begin{lm}\label{trivial Lemma}
Let $\KC$ be a \Fraisse{} class that does not split and let $\MB$ be the \Fraisse{} limit of $\KC$. Let also $c\in M$. Then:
\begin{enumerate}
\item there is a finite $K\subset M$ so that $\KB$ controls $c$; \label{trivial Lemma 1}
\item if $\KB$ controls $c$, $L$ is a finite subset of $M$ with $K\subset L$ and $c\not\in L$ then $\LB$ also controls $c$; \label{trivial Lemma 2}
\item if $\KB$ controls $c$ and $f\in\Aut{}(\MB)$, then $\langle f(K)\rangle_{\MB}$ controls $f(c)$. \label{trivial Lemma 3}
\end{enumerate}
\end{lm}
\begin{proof}
All three statements follow directly from the definition \ref{DefDoesNotSplit}.
\end{proof}

\begin{lm}\label{permutationz}
Let $\KC$ be a \Fraisse{} class that does not split and let $\MB$ be the \Fraisse{} limit of $\KC$. Assume $F\subset M$ is  finite with $K\subset F$ so that $\KB$ controls the points $c_1,\ldots,c_n\in F^c$. Let $p_i$ be the rqf-type of $c_i$ over $\KB$. If $f:F\cup\{c_1,\ldots,c_n\}\to M$ is an injective map with so that $f\res F$ in an embedding of $\FB$ in $\MB$ and the rqf-type of $f(c_i)$ over $f(K)$ is $f[p_i]$. Then $f$ is also an embedding.
\end{lm}
\begin{proof}
We will prove this by induction. For $n=1$, let $g\in\Aut{}(\MB)$ with $g\res_{K\cup\{c_1\}}=f\res{}_{K\cup\{c_1\}}$. The map $h:F\cup\{c_1\}\to g^{-1}f\big(F\cup\{c_1\}\big)$ is an injection fixing $K\cup\{c_1\}$ with $h\res F$ being an isomorphism. Since $\KB$ controls $c_1$, the map $h$ (and therefore the map $f$) is an isomorphism.

Assume now that the statement holds for every $n$ with $n\leq k$ and let $f:F\cup\{c_1,\ldots,c_k,c_{k+1}\}\to M$. By the inductive hypothesis we can enlarge $F$ to include $\{c_1,\ldots,c_k\}$, reducing the problem again to the $n=1$ case. 
\end{proof}

 For $a,b\in M$, we say that $a$ is \emph{equivalent} to $b$ and we write $a\sim_{\MB} b$, if there is a finite $K\subset M$ such that $\KB$ controls $a$ and $a,b$ have the same rqf-type over $\KB$. Notice then, that as a consequence of \ref{trivial Lemma}(\ref{trivial Lemma 3}) $\KB$ controls $b$ too.

\begin{lm}\label{sameQftypeOverF}
Let $\KC$ be a \Fraisse{} class that does not split and let $\MB$ be the \Fraisse{} limit of $\KC$. Let also $a,b\in M$ with $a\sim_{\MB} b$. Then, for every finite subset $F$ of $M$ with $a,b\not\in F$, the points $a$ and $b$ share the same rqf-type over $\FB$.
\end{lm}
\begin{proof}
Let $f:F\cup\{a,b\}\to F\cup\{a,b\}$ be the function that fixes $F$ and exchanges $a$ with $b$ and use Lemma \ref{permutationz}.
\end{proof}

In the additional presence of $\mathrm{SAP}$ we now have the following results. 

\begin{lm}\label{LemmaCommon L}
Let $\KC$ be a \Fraisse{} class with $\mathrm{SAP}$ that does not split and let $\MB$ be the \Fraisse{} limit of $\KC$.
Let also $c_1,\ldots c_n\in M$. Then there is $K\subset M$ so that $\KB$ controls $c_i$ for every $i\in\{1,\ldots,n\}$.
\end{lm}
\begin{proof}
From Lemma \ref{trivial Lemma}(\ref{trivial Lemma 1}) let $K_i\subset M$ so that $\KB_i$ that controls $c_i$. Using SAP and \ref{trivial Lemma}(\ref{trivial Lemma 3}), we can arrange $K_1,\ldots,K_n$ in such a way that $K_1,\ldots,K_n,\{a,b,c\}$ are all pairwise disjoint. Let $K=\bigcup_{i=1}^n K_i$
\end{proof}

\begin{coro}\label{lemmaEquivalence}
Let $\KC$ be a \Fraisse{} class with $\mathrm{SAP}$ that does not split and let $\MB$ be the \Fraisse{} limit of $\KC$.
Then, $\sim_{\MB}$ is  an equivalence relation on $M$.
\end{coro}
\begin{proof}
Reflexivity and symmetry follow directly from the definition. Transitivity, follows from Lemma \ref{LemmaCommon L} and Lemma \ref{sameQftypeOverF}.
\end{proof}

\begin{lm}\label{theotherdirection}
Let $\KC$ be a \Fraisse{} class with $\mathrm{SAP}$ that does not split and let $\boldsymbol{M}$ be the \Fraisse{} limit of $\KC$. Let also $A,B\subset M$ such that all $\AB,\AB^c,$ and $\BB,\BB^c$, absorb points. Then, every $g$ in $\mathrm{Iso}(\boldsymbol{A},\boldsymbol{B})$  can be extended to an automorphism $\tilde{g}\in \mathrm{Aut}(\boldsymbol{M})$.
 \end{lm}
\begin{proof}
Let $\{N_i: i\in I\}$ be an enumeration of all equivalence classes of $\sim_{\MB}$ in  $M$ and notice that since all $\AB,\AB^c,\BB,\BB^c$ absorb points the sets $N_i\cap A,N_i\cap A^c,N_i\cap B,N_i\cap B^c$ are all infinite for every $i\in I$.

Let now $c\in A^c$ and pick  $K\subset M$ so that $\KB$ controls $c$. Using Lemma \ref{trivial Lemma}(\ref{trivial Lemma 3}) and the fact that $\AB$ absorbs points we can assume that $K\subset A$. Notice then, that by Lemma \ref{sameQftypeOverF} $\KB$ controls every other point $c'\in M$ with $c\sim_{\MB}c'$ and $c'\not\in K$. In particular, if $c\in N_i$ for some $i\in I$, $\KB$ controls every point $c'\in N_i\cap A^c$ as well as cofinitely many points $c'\in N_i\cap A$. So, for every $i\in I$ we can pick a finite subset $K_i$ of $A$ and some $a_i\in N_i\cap A$ so that $\KB_i$ controls $a_i$ as well as every $c\in N_i\cap A^c$.   

Given now any $g\in\mathrm{Iso}(\boldsymbol{A},\boldsymbol{B})$ we have by Lemma \ref{trivial Lemma}(\ref{trivial Lemma 3}) that $\langle g(K_i)\rangle_{\MB}$ controls $g(a_i)$. For every $i\in I$, let $gi$ be the unique $j\in I$ with $g(a_i)\in N_j$ and pick a bijection $h_i:N_i\cap A^c\to N_{gi}\cap B^c$. We extend $g$ to an automorphism $\tilde{g}\in \mathrm{Aut}(\boldsymbol{M})$ setting $\tilde{g}(c)=h_i(c)$ whenever $c\in A^c$ with $c\in N_i$. To see that $\tilde{g}$ is indeed an automorphism, notice that by Lemma \ref{permutationz} the restriction of $\tilde{g}$ to any finite substructure is a partial isomorphism.
\end{proof}

 \begin{theorem}\label{theorem}
Let $\boldsymbol{M}$ be a \Fraisse{} structure that has no algebraicity and let $\mathcal{K}$ be the corresponding \Fraisse{} class. 
\begin{enumerate}
\item If $\mathcal{K}$ splits, then the generic substructure $\AB$ of $\MB$ is a strongly non-global substructure and the generic pair $\AB,\BB$ in $\MB$ is a strongly non-global pair.
\item If $\mathcal{K}$ does not split, then  the generic substructure $\AB$ of $\MB$  is a global substructure and the generic pair $\AB,\BB$ in $\MB$  is a global pair.
\end{enumerate} 
 \end{theorem}
\begin{proof}
We have from Lemma \ref{pair} that for a generic substructure $\AB$ of $\MB$, both $\AB,\AB^c$ absorb points and that for a generic pair $\AB,\BB$ in $\MB$, all $\AB,\AB^c,\BB,\BB^c$ absorb points. The result follows from Lemma \ref{not extend} in case that $\KC$ splits and from Lemma \ref{theotherdirection} in case that $\KC$ does not split.   
\end{proof}

\section{Some structural consequences}\label{ChapterStructuralConsequences}

Let $\KC$ be a \Fraisse{} class with SAP, and let $\MB$ be the corresponding \Fraisse{} limit. If $A$ is a subset of $M$, we denote with $S^{\MB}_n(A)$ the Stone space of all (complete) $n$-types over $A$. For every finite subset $C$ of $M$ we denote by $\mathrm{tp}(C|A)$ the type of $C$ over $A$. Let now  $C_0\subset M$ so that $\CB_0$ splits $\KC$ and let $n$ be the size of $C_0$. The fact that $\CB_0\in\KC$ splits $\KC$ can be rephrased as follows: for every finite $K\subset M$ with $K\cap C_0=\emptyset$ there is a finite $F\supset K$ and a second copy $\CB_1$ of $\CB_0$ in $\MB$ such that $F\cap C_0=\emptyset$, $F\cap C_1=\emptyset$ and $\mathrm{tp}(C_0|F)\neq\mathrm{tp}(C_1|F)$. Iterating this fact we produce a Cantor schema in the compact metric space $S^{\MB}_n(M)$ which results to an embedding of the Cantor set $2^\N$ into $S^{\MB}_n(M)$. We have just proved the following proposition.
 \begin{propo} \label{theorem omega stable--> does not split }
 Let $\KC$ be a \Fraisse{} class with SAP. If the corresponding \Fraisse{} limit $\MB$ is $\omega$-stable, then $\KC$ does not split.
\end{propo}  
Someone would hope that the above result could turn into a characterization of $\omega$-stability for \Fraisse{} limits without algebraicity. This however is not the case as the following example exhibits.
\begin{example}\label{Example not omega stable}
Let $\LC=\{R,S\}$ where $R,S$ are both binary relational symbols and let $M$ be the disjoint union of a countable family of countable sets $\{N_i : i\in\N\}$. We define $\MB$ to be an $\LC$-structure with domain $M$, where the symbols of $\LC$ are interpreted as follows. For $S$, let 
\[\MB\models S(a,b)\Leftrightarrow \exists i\in\N\,\,\, a,b\in N_i\]
To interpret $R$, equip first the set of indeces $\{i: i\in\N\}$ with a structure $\mathbb{G}$ isomorphic to the random graph in the language $\LC=\{R'\}$ of one binary symbol. Let
\[\MB\models R(a,b)\quad\Leftrightarrow\quad a\in N_i,\,\,\,b\in N_j\,\,\,\text{and}\,\,\,\mathbb{G}\models R'(i,j)\] 
It is not difficult to see that $\MB$ is a \Fraisse{} limit without algebraicity. For every $a\in M$, $a$ is controlled  by $\langle b\rangle_{\MB}$ for any $b\neq a$ that lies in the same $N_i$ with $a$. From that if follows that the corresponding \Fraisse{} class $\KC$ does not split. Notice however, that $\MB$ is not $\omega$-stable because the structure $\mathbb{G}$ of the random graph is not $\omega$-stable.
\end{example}   

We will now see that Example \ref{Example not omega stable} is an archetype of how \Fraisse{} limits of classes $\KC$ which have $\mathrm{SAP}$ and do not split look like. Recall that in Chapter \ref{ChapterSAP} we defined a relation $\sim_{\MB}$ between points a \Fraisse{} limit $\MB$ whose \Fraisse{} class $\KC$ does not split. If moreover $\KC$ has SAP, we proved in Lemma \ref{lemmaEquivalence} that $\sim_{\MB}$ is an equivalence relation on $M$. Another useful observation is that $\sim_{\MB}$ is $\Aut{}(\MB)$-invariant. This follows directly from Lemma \ref{trivial Lemma}(\ref{trivial Lemma 3}).

Let $\{N_i : i\in I\}$ be the partition of $M$ into  the equivalence classes of $\sim_{\MB}$, where $I$ is a countable, possibly finite set of indices . From the fact that $\KC$ has SAP it is straight forward that for each $i\in I$, $N_i$ is infinite. 
Since $\sim_{\MB}$ is $\Aut{}(\MB)$-invariant, we have a natural action of $\Aut{}(\MB)$ on the set of indices $I$: for every $i,j\in I$ let
\[g\cdot i = j\quad\Leftrightarrow\quad\exists a\in N_i\,\,\,g(a)\in N_j\quad\Leftrightarrow\quad\forall a\in N_i\,\,\,g(a)\in N_j.\]
Let $G_0$ be the kernel of the action $\Aut{}(\MB)\curvearrowright I$ and let $H=\Aut{}(\MB)/G_0$. $H$ is a subgroup of $\mathrm{S}_I$, group of all permutations on the set $I$. From the analysis above it follows that the automorphism group $\Aut{}(\MB)$ is a subgroup of the unrestricted Wreath product $G=\mathrm{S}_\infty \,\mathrm{Wr}_I\, H = \big(\prod_{i\in I}\mathrm{S}_\infty\big)\rtimes H$ where the $i$-th copy of $\mathrm{S}_\infty$ is the permutation group of $N_i$; see for example \cite{oligomorphic}. Moreover, by Lemma \ref{permutationz} it follows that  $\Aut{}(\MB)$ lies densely in $G$. Therefore, since $\Aut{}(\MB)$ is a closed subgroup of $\mathrm{S}_\infty$, the groups $\Aut{}(\MB)$ and $G$ are actually equal.

So, we have shown that if $\MB$ is the \Fraisse{} limit of a class $\KC$ that has SAP and does not split then,  $\Aut{}(\MB)=\big(\prod_{i\in I}\mathrm{S}_\infty\big)\rtimes H$ where $H$ is a subgroup of $\mathrm{S}_I$ and $I$ is finite or countably infinite. Notice also that $H$ is a closed subgroup of $\mathrm{S}_I$. Working now towards the opposite direction, let $I$ be a finite or countably infinite set and let $H$ be a closed subgroup of $\mathrm{S}_I$. Let $\boldsymbol{I}_H$ be the canonical $\LC_H$-structure with domain $I$ where the language $\mathcal{\LC}_H=\{R^n_{j}\}$ has one distinct $n$-ary symbol $R^n_{j}$ for each orbit $\mathcal{O}_j$ of the action of $H$ on $I^n$; see for example \cite{oligomorphic}.

Let $\LC=\LC_H\cup\{S\}$, where $S$ is a new binary symbol and let $M=I\times\N$. Consider the $\LC$-structure $\MB_H$ on $M$, where the interpretation is done as follows.

For every for every $R^n_{j}\in\LC_H$ we have
\[\MB_H\models R^n_{j}\big((i_1,m_1),\ldots,(i_n,m_n)\big)\quad\Leftrightarrow\quad\boldsymbol{I}_H\models R^n_{j}(i_1,\ldots,i_n)\]
and
\[\MB_H\models S\big((i_1,m_1),(i_2,m_2)\big)\quad\Leftrightarrow\quad i_1=i_2.\]

Let also $\KC_H=\mathrm{Age}(\MB_H)$. It is easy to check that $\MB_H$ is a countable, ultrahomogeneous structure without algebraicity and that $\Aut{}(\MB_H)=\big(\prod_{i\in I}\mathrm{S}_\infty\big)\rtimes H$. To see that $\KC_H$ does not split let $C\subset M$ finite with $C=\{(i_1,m_1),\ldots,(i_k,m_{k})\}$ where, some ordered couples might share the same index $i\in I$. Let $m=\mathrm{max}\{m_j+1 : 1\leq j\leq k\}$ and let $K=\{(i_1,m),\ldots,(i_k,m)\}$. Then, if $D=K\cup C$, the inclusion embedding $i:\CB\to\DB$ shows that $\CB$ does not split $\KC$.

We collect in the following theorem the above results:
\begin{theorem}\label{theorem_Wreath_Product}
Let $\MB$ be a \Fraisse{} structure without algebraicity. If $\mathrm{Age}(\MB)$ does not split then $\Aut{}(\MB)=\big(\prod_{i\in I}\mathrm{S}_\infty\big)\rtimes H$, where $I$ is a countable possibly finite set of indeces, $H$ is a closed subgroup of $\mathrm{S}_I$, and $H$ acts on $I$ in the natural way. Moreover, if $G$ is a  group isomorphic to $\big(\prod_{i\in I}\mathrm{S}_\infty\big)\rtimes H$ where $I$ and $H$ as above, then, there is a countable ultrahomogeneous  structure without algebraicity $\MB_H$ such that $\mathrm{Age}(\MB_H)$ does not split and such that $G$ is isomorphic to $\mathrm{Aut}(\MB_H)$.
\end{theorem}

The question of whether the complete inverse of the first statement of Theorem \ref{theorem_Wreath_Product} holds, remains open. 
\begin{question}
 Let $\MB$ be a \Fraisse{} structure  without algebraicity. Assume moreover that $\Aut{}(\MB)$ is isomorphic as a topological group to $\big(\prod_{i\in I}\mathrm{S}_\infty\big)\rtimes H$, where $H$ and $I$ as above. Is it the case that $\mathrm{Age}(\MB)$ does not split?
\end{question}
\section{The Wijsman hyperspace topology}\label{ChapterWijsman}

In Chapter \ref{ChapterSAP} and Chapter \ref{ChapterStructuralConsequences} we worked with countable relational structures. Given such a structure $\MB$, we  viewed the set of all substructures of $\MB$ as a Polish space. Namely, the cantor space $2^M$. In the next chapter we are going to work with the Urysohn space and the Urysohn sphere. Both are complete metric space of the size of the continuum. Like in Chapter \ref{ChapterSAP}, we will need a natural Polish space whose elements correspond to the substructures of the space under consideration.    

Let $(X,d)$ be some metric space and let $\mathcal{F}(X)$  be the set of all closed subsets of $X$. 
The Wijsman topology on $\mathcal{F}(X)$, introduced in \cite{Wijsman}, is the weakest topology on $\mathcal{F}(X)$ that makes continuous the family of distance functionals $\{d_x\}_{x\in X}$, where $d_x:\mathcal{F}(X)\to\mathbb{R}$ with $d_x(F)=d(x,F)$.  The proof of the following theorem can be found in \cite{Beer}.

\begin{theorem}[Beer]
If the space $(X,d)$ is complete and separable, then the Wijsman topology of $\mathcal{F}(X)$ is Polish.
\end{theorem}

 From now on $\mathcal{F}(X)$ will always be equipped with the Wijsman topology.  We will say that for a \emph{generic} subspace of $X$ a certain property holds if the set $\mathcal{A}\subset\mathcal{F}(X)$ of all $F$ which have this property  is comeager  in $\mathcal{F}(X)$. Similarly we say that for a \emph{generic} pair of subspaces of $X$ a certain property holds if the set $\mathcal{A}\subset\mathcal{F}(X)\times\mathcal{F}(X)$ of all $(A,B)$ which have this property is comeager  in $\mathcal{F}(X)\times\mathcal{F}(X)$.
 
\section{The Urysohn space $\U$ and the Urysohn sphere $\mathbb{S}$}\label{ChapterUrysohn}

Our aim here is to prove Theorem \ref{UrysohnMainTheorem}. The (separable) \emph{Urysohn metric space} $(\mathbb{U},\rho)$ is the unique, up to isometry, Polish metric  space that satisfies the following properties:
\begin{itemize}
\item[$\cdot$] (ultrahomogeneity) for every two finite isometric subspaces  $A,B\subset\mathbb{U}$ and for every isometry $f:A\to B$, $f$ extends to a full isometry $\tilde{f}$ of $\mathbb{U}$;
\item[$\cdot$] (universality) every Polish metric space is isometric to a subspace of $\mathbb{U}$. 
\end{itemize}    
The Urysohn space was introduced by Urysohn in \cite{Urysohn} but the interest in this space was revived through the work of Kat\v{e}tov and Uspenskij \cite{Katetov, Uspenskij}. Here, in the next couple of paragraphs, we record some definitions and facts commonly used in the study of Urysohn space. For a more detailed exposition the reader may want to consult  Melleray (\cite{MellerayInitial} or \cite{Melleray}).

Let $(A,d_A),(B,d_B)$ be isometric Polish metric spaces. We are going to denote with $\Iso(A,B)$ the space of all (bijective) isometries from $A$ to $B$ and with $\Iso(A)$ the space of all (bijective) isometries from $A$ to $A$. A basic open set $U$ in  $\Iso(A,B)$ can be thought of as a couple $[f,\delta]$, where $f$ is a partial isometry from $A$ to $B$ with $\mathrm{dom}f$ finite and $\delta>0$. If $U\subset\Iso(A,B)$ is a basic open set corresponding to the couple $[f,\delta]$, then $g\in \Iso(A,B)$ belongs to $U$ if for every $a\in \mathrm{dom}f$ we have $d_B(f(a),g(a))<\delta$. 
 
\begin{definition}\label{katetov def}
Let $(X,d)$ be a metric space. A map $g:X\to\mathbb{R}$ is a \emph{Kat\v etov map} on $X$ if 
\[\forall x, y \in X\,\,\, g(x)-g(y)\leq d(x,y) \leq g(x)+g(y).\] 
\end{definition}
If moreover $\mathrm{range}(g)\subset\mathbb{Q}$,  we say that $g$ is a \emph{rational Kat\v etov map}  on $X$. We denote by $E(X)$  the set of all Kat\v etov maps  on $X$, and by  $E^{\mathbb{Q}}(X)$ the set of all rational Kat\v etov maps on $X$.
Following Kat\v etov, we introduce a distance between any pair of maps $g_1,g_2$ that belong to $E(X)$, given by
\[d_X(g_1,g_2)=\sup\{|g_1(x)-g_2(x)|: x\in X\}.\]
This distance renders $(E(X),d_X)$ a complete metric space which extends $(X,d)$ via the identification $x\to g_x(y)=d(x,y)\in E(X),$ for every $x\in X$. We are interested in a specific subset of $E(X)$ which plays important role in the study of Urysohn space. Let $Y$ be a subset of $X$ and $g$ a Kat\v etov map  on $Y$. We can extend $g\in E(Y)$ to $\tilde{g}\in E(X)$ by letting
\[\tilde{g}(x)=\mathrm{inf}\{g(y)+d(x,y): y\in Y\}.\]
We call $\tilde{g}$ the \emph{Kat\v etov extension} of $g$ to $X$. If $f \in E(X)$ and $Y\subset X$ are such that $f$ is the Kat\v etov extension of $f\res_Y$ to $X$, we say that $Y$ is a support of $f$. The set
\[E(X,\omega)=\{g\in E(X): g \,\,\text{has finite support} \}\]
should be thought of as the set of all the rqf-types over every finite substructure of a structure that we saw in Chapter \ref{ChapterSAP}.  Kat\v{e}tov maps on $X$ with finite support can be approximated by functions that belong to   
\[E^{\mathbb{Q}}(X,\omega)=\{g\in E^{\mathbb{Q}}(X): g \,\,\text{has finite support} \}.\] 
Moreover, if $X$ is separable and $D\subseteq X$ is countable and dense in $X$, then $E^{\mathbb{Q}}(D,\omega)$ is a countable dense subset of $E(X,\omega)$.

There is another useful characterization of Urysohn space. Let  $X$ be a metric space. We say that $X$ has the \emph{approximate extension property}, if for every finite $A$ subset of $X$, every $g\in E(A)$ and every $\varepsilon>0,$ there is a $z\in X$ such that for every $a\in A$ we have $|d(z,a)-g(a)|\leq\varepsilon$. We say that $X$ has the \emph{extension property} if we can take $\varepsilon=0$ in the above definition. For a complete separable metric space $X$, the following are equivalent:
\begin{itemize}
\item[$\cdot$] $X$ is isometric to the Urysohn space.
\item[$\cdot$] $X$ has the extension property.
\item[$\cdot$] $X$ has the approximate extension property.
\end{itemize}
Uspenskij in \cite{Uspenskij}, used Kat\v{e}tov's tower construction \cite{Katetov} to prove that for every Polish metric space $X$ we can find subspaces $B$ of $\mathbb{U}$ isometric to $X$ for which every isometry of $B$ extends to  a full isometry of $\mathbb{U}$. Therefore, there are subspaces of $\U$, of infinite cardinality, which are global. 
The following example shows that $\U$ contains non-global subspaces too (see also \cite{MellerayInitial} for non-global embeddings of every non-compact space). 
\begin{example}\label{Example_not_free_subspace_of_Urysohn}
Let $x_0,x_1,x_2$ be points in $\U$ such that $\rho(x_0,x_1)=2, \rho(x_1,x_2)=1, \rho(x_0,x_2)=3$ let also $\e$ with $0<\e<1$. Set $A=\U\stm B(x_0,\e)$. Then it is easy to see that $A$ has the extension property and therefore it is isomorphic to $\U$. So, there is an isometry $f:A\to A$ sending $x_1$ to $x_2$. This isometry cannot be extended to a full isometry of $\U$. Moreover, the choice of $\e$ ensures that there is a uniform lower bound bigger than zero  (say $1-\e$) between $g\res A$ and $f$ for every isometry $g$ of $\U$.     
\end{example}

\begin{definition}\label{definitionAbsorbsUrysohn}
Let $D\subseteq\mathbb{U}$. We say that $D$ \emph{absorbs points} if $D\neq\emptyset$, and for every finite $X=\{x_1,\ldots,x_k\}\subset\mathbb{U}$, for every $g\in E(X)$ with $\rho(x_i,D)<g(x_i)$  there is a $z\in D$ such that $\rho(z,x_i)=g(x_i)$.
\end{definition}
\begin{rem} \label{Remark absorbs-->Ury}
Notice that if $F$ is a closed subset of $\mathbb{U}$ and $F$ absorbs points, then $F$ is complete, separable and has the extension property. Therefore it is isomorphic to the Urysohn space. 
\end{rem}

\begin{lm}\label{TheoremUrysohnAbsorbsPoints}
For a generic $F\in\mathcal{F}(\mathbb{U})$, $F$ 
absorbs points.
\end{lm}
\begin{proof}

Let $P$ be a countable dense subset of $\mathbb{U}$. Consider  the set
\begin{eqnarray*}
\mathcal{A}=\big\{F\in\mathcal{F}(\mathbb{U}): \forall\,\,\text{finite}\,\,A=\{a_1,\ldots,a_k\}\subset P \quad\forall g\in E^\mathbb{Q}(A)\quad\forall m\in\mathbb{N}\\
\rho(a_1,F)\geq g(a_1)\vee\ldots\vee \rho(a_k,F)\geq g(a_k)\quad\text{or}\quad\exists z\in U\\
\text{such that}\,\,\,\rho(a_1,z)=g(a_1)\,\wedge\,\ldots\,\wedge\, \rho(a_k,z)=g(a_k)\,\wedge\, \rho(z,F)<\frac{1}{m}\}.
\end{eqnarray*}
Closed conditions in a Polish space are also $G_\delta$. So $\mathcal{A}$  a $G_\delta$ subset of $\mathcal{F}(\mathbb{U})$. We now show that $\mathcal{A}$ is also dense in $\mathcal{F}(\mathbb{U})$. Let $F\in\mathcal{F}(\mathbb{U})$. We will find a sequence $\{F_n\}_{n\in\N}$ of sets from $\mathcal{A}$ that converges to $F$ in the Wijsman topology. If $F=\mathbb{U}$, then the sequence $F_n=\mathbb{U}$ for every $n$ lies in $\mathcal{A}$ and converges to $\mathbb{U}$. If $F\neq\mathbb{U}$, let $\{x_n\}_{n\in\mathbb{N}}$ be a dense subset of $F^c$ and   $d_n=\rho(x_n,F)>0.$ Let 
\[F_n=\big\{x\in\mathbb{U}: \rho(x,x_1)\geq d_1\cdot(1-\frac{1}{n}) ,\ldots, \rho(x,x_n)\geq d_n\cdot(1-\frac{1}{n}) \big\}.\] 
Clearly $\{F_n\}_{n\in\N}$ converges to $F$. Fix now a $n\in\mathbb{N}$. We will show that $F_n$ belongs to $\mathcal{A}$. Let $A=\{a_1,\ldots,a_k\}\subset P$,  $g\in E^\mathbb{Q}(A)$ and $m\in\mathbb{N}$. Assume moreover that $\rho(a_i,F)<g(a_i)$ for every $i$. Let $\tilde{g}$ be the Kat\v{e}tov extension of $g$ to $A\cup\{x_1,\ldots x_n\}$. The extension property of the Urysohn space gives us a point $z\in\mathbb{U}$ that realizes the distances given by $\tilde{g}$.  Moreover, for every $j\in\{1,\ldots,n\}$ we have that
\begin{eqnarray*}
\tilde{g}(x_j)=\mathrm{inf}\{\rho(x_j,a_i)+g(a_i) : 1\leq i\leq k\}>\\
\mathrm{inf}\{\rho(x_j,a_i)+\rho(a_i,F) : 1\leq i\leq k\}\geq\\ \rho(x_j,F)> d_j\cdot(1-\frac{1}{n}).
\end{eqnarray*}
Therefore $0=\rho(z,F_n)<\frac{1}{m}$, which proves that $F_n$ belongs to $\mathcal{A}$ for every $n\in\mathbb{N}$.
  
We showed so far that $\mathcal{A}$ is a comeager subset of $\mathcal{F}(\mathbb{U})$. We now  proceed to prove that for every $F$ that belongs to $\mathcal{A}$, $F$ absorbs points. Let $F\in\mathcal{A}$ and $X=\{x_1,\ldots,x_k\}\subset\mathbb{U}$. Let also $g\in E(X)$, with $\rho(x_i,F)<g(x_i)$ for every $i\in\{1,\ldots,k\}$. Urysohn space is complete, so it suffice to find a sequence $\{z_n\}_{n\in\N}$ in $\mathbb{U}$ such that for every $n\in\mathbb{N}$ we have:
\begin{itemize}
\item[$\cdot$] $|\rho(z_n,x_i)-g(x_i)|<2^{-n}$ for every $i\in\{1\ldots,k\}$;
\item[$\cdot$] $\rho(z_{n},F)<2^{-n}$, and
\item[$\cdot$] $\rho(z_{n+1},z_n)<2^{1-n}$.
\end{itemize} 
Let  $d=\mathrm{min}\{g(x_i)-\rho(x_i,F) : 1\leq i\leq k\}>0$ and let $\{\delta_n\},\{\delta'_n\}$ be two sequences of positive real numbers with $\delta_n,\delta'_n<\mathrm{min}\{d,1\}\cdot 2^{-(n+1)}$. For $n=1$ let $P_1=\{p_1^1,p_2^1,\ldots,p_k^1\}\subset P$ with $\rho(p_i^1,x_i)<\delta_1$ for all $i\in\{1,\ldots k\}$ and let $g_1\in E^\mathbb{Q}(P_1)$ with $|g_1(p^1_i)-g(x_i)|<\delta'_1$ for all $i\in\{1,\ldots,k\}$. These conditions imply that $g_1(p_i^1)>\rho(p_i^1,F)$. Therefore, by definition of $\mathcal{A}$, we can find a $z_1$ such that
\begin{itemize}
\item[$\cdot$] $|\rho(z_1,x_i)-g(x_i)|\leq \delta_1+\delta'_1<2^{-1}$ for every $i\in\{1\ldots,k\}$, and
\item[$\cdot$] $\rho(z_{1},F)<2^{-1}$.
\end{itemize}
Suppose now that we have defined $z_1,\ldots,z_n$ fulfilling the above properties. Let $f_n\in  E(\{x_1,\ldots,x_k\})$ with $f_n(x_i)=d(x_i,z_n)$. Then $d_X(f_n,g)=\mathrm{sup}\{|f_n(x_i)-g(x_i)| :1\leq i\leq k\}<2^{-n}$.

We define now $P_{n+1}=\{p_1^{n+1},p_2^{n+1},\ldots,p_k^{n+1}\}\cup\{p_*^{n+1}\}\subset P$ with $\rho(p_*^{n+1},z_{n})<\delta_{n+1}$, $\rho(p_i^{n+1},x_i)<\delta_{n+1}$ for all $i\in\{1,\ldots k\}$  and $g_{n+1}\in E^\mathbb{Q}(P_{n+1})$ with $|g_{n+1}(p^{n+1}_*)-d_X(f_n,g)|<\delta'_{n+1}$ and $|g_{n+1}(p^{n+1}_i)-g(x_i)|<\delta'_{n+1}$ for all $i\in\{1,\ldots,k\}.$
 Again these conditions imply that $g_{n+1}(p_i^{n+1})>\rho(p_i^{n+1},F)$ so we get a $z_{n+1}$
such that
\begin{itemize}
\item[$\cdot$] $|\rho(z_{n+1},x_i)-g(x_i)|\leq \delta_{n+1}+\delta'_{n+1}<2^{-(n+1)}$ for every $i\in\{1\ldots,k\}$;
\item[$\cdot$] $\rho(z_{n+1},F)<2^{-(n+1)}$, and
\item[$\cdot$] $\rho(z_{n+1},z_{n})\leq d_X(f_n,g)+\delta_{n+1}+\delta'_{n+1}<2^{-n}+2^{-(n+1)}<2^{1-n}.$
\end{itemize}
Which proves that  every $F\in\mathcal{A}$ absorbs points and therefore, the set of all closed subsets $F$ of $\mathbb{U}$ that 
absorb points 
is a comeager subset of $\mathcal{F}(\mathbb{U})$.   
\end{proof}
By the Lemma \ref{TheoremUrysohnAbsorbsPoints} and Remark \ref{Remark absorbs-->Ury} we have the following theorem.
\begin{theorem}\label{Theoerem Generic is Urysohn}
For a generic $F\in\mathcal{F}(\mathbb{U})$, $F$ is isometric to $\U$.
\end{theorem}
With the following lemma we establish in relation to the Definition \ref{splits}, that the class of all finite metric spaces ``splits'' in a uniform way.

\begin{lm}\label{lemma g1 g2}
Let $d_*$ be a positive real number. Then for every $\varepsilon<d_*$ and every finite metric space $(X,d)$ with $X=\{a_1,\ldots,a_n,c\}$ such that $d(c,a_i)\geq d_*$, for all $i\in\{1,\ldots,n\}$, there are $g_1,g_2\in E(X)$ such that 
\begin{itemize}
\item[$\cdot$]  $g_1(a_i)=g_2(a_i)$ for all $i\in\{1,\ldots,n\}$;
\item[$\cdot$]  $g_1(c),g_2(c)\geq d_*$, and
\item[$\cdot$] $|g_1(c)-g_2(c)|>\varepsilon$.
\end{itemize}
\end{lm}
\begin{proof}
Let $D=\mathrm{diam}(X)$ and pick any $\delta\in(\e,d_*)$. Define $g_1,g_2:X\to\mathbb{R}$ with $g_1(a_i)=g_2(a_i)=2D$ for every $i\in\{0,\ldots,n\}$, $g_1(c)=2D$ and $g_2(c)=2D-\delta$. 
\end{proof}

\begin{lm}\label{LemmaUrysohnNotExtend}
Let $A,B\in\mathcal{F}(\mathbb{U})$ such that $A,B$ absorb points and $A,B\neq\mathbb{U}$. Then, $\mathcal{E}(A,B)$ is a meager subset of $\mathrm{Iso}(A,B)$.
\end{lm}
\begin{proof}
The proof is essentially the same as in Lemma \ref{not extend}. Therefore, we will skip the details. Players I and II take turns playing open sets in the Banach Mazur game $G^{**}(\mathcal{N},\mathrm{Iso}(A,B))$ where 
$\mathcal{N}=\mathrm{Iso}(A,B)\setminus\mathcal{E}(A,B)$. Let $c_*\in A^c$, let $d_*\in\mathbb{R}$ with $0<d_*<\rho(c_*,A)$ and let $\varepsilon$ with $0<\varepsilon<d_*$. Consider the set 
\[C=\{y\in B^c : \exists h\in\mathcal{E}(A,B)\,\,\text{with}\,\,h(c_*)=y \},\]
and let $\{W_i\}_{i\in\mathbb{N}}$ be an open covering of $C$ with $\mathrm{diam}(W_i)<\frac{\varepsilon}{2}$ for every $i\in\mathbb{N}$. 

Assume that in the $n$-th step, Player I has played the open set $U_n=[f_n,\delta_n]\subset\mathrm{Iso}(\mathbb{U})$, where $f_n$ is an isometry between the finite subspaces $A_n\subset A$ and $B_n\subset B$. Assume also that $i_n$ is the smallest index for which there is an  $h\in\mathcal{E}(A,B)$ with $h(c_*)\in W_{i_n}$ and let  $y_n=h(c_*)$ for any $h$ as above. Let also $h_n:A_n\cup\{c_*\}\to B_n\cup\{y_n\}$ be the unique partial isometry that extends $f_n$ to $A_n\cup\{c_*\}$.

Player II will make use of  Lemma $\ref{lemma g1 g2}$ to get $g_1,g_2\in E(A_n\cup\{c_*\})$ with $g_1\res_{A_n}=g_2\res_{A_n}$ and $|g_1(c_*)-g_2(c_*)|>\varepsilon$. Moreover, by adding if necessary the same constant function to $g_1$ and $g_2$, he can arrange so that $g_i(c_*)>\rho(c_*,A)$ for $i\in\{1,2\}$. Due to the fact that both $A$ and $B$ absorb points, he will find points $z_n\in A$ and $z'_n\in B$ such that $\rho(x,z_n)=g_1(x)$ for every $x\in A_n\cup\{c_*\}$ and   $\rho(x,z'_n)=g_2\circ h^{-1}_n(x)$ for every $x\in B_n\cup\{y_n\}$. Player II will play his $n$-th move $V_n=[f'_n,\delta'_n]$ where 
\[f'_n:A_n\cup\{z_n\}\to B_n\cup\{z'_n\}\quad\text{with}\quad f'_n\res_{A_n}=f_n,\]
$f'_n(z_n)=z'_n$ and $\delta'_n=\mathrm{min}\{\delta_n,\frac{\varepsilon}{2},2^{-n}\}$. As in  Lemma \ref{not extend} this leads to a winning strategy for Player II and by Theorem \ref{banach-mazur} we have that $\mathcal{E}(A,B)$ is a meager subset of $\mathrm{Iso}(A,B)$.
\end{proof}

Summarizing  the above results we have the following theorem:

\begin{theorem}\label{UrysohnMainTheorem}
Let $\U$ be the Urysohn space. Then for a generic $F\in\mathcal{F}(\U)$, the generic isometry $f\in \Iso(F)$ cannot be extended to an isometry $\tilde{f}\in \Iso(\U)$. Moreover, for a generic pair $(A,B)\in\mathcal{F}(\U)^2$, the generic isometry $f\in \Iso(A,B)$ cannot be extended to an isometry $\tilde{f}\in \Iso(\U)$. 
\end{theorem}
\begin{proof}
By Theorem \ref{TheoremUrysohnAbsorbsPoints}, we have that for a generic subspace $F\in\mathcal{F}(\U)$ and for a generic pair $A,B\in\mathcal{F}(\U)$ all $F,A,B$ absorb points. Moreover, for generic $F,A,B$,  $F,A,B\neq\U$. Lemma \ref{LemmaUrysohnNotExtend} proves the rest.
\end{proof}

We want to point out here that the situation with the Urysohn sphere $\mathbb{S}$ is not much different. Analogous statements to Lemmas  \ref{TheoremUrysohnAbsorbsPoints}, \ref{lemma g1 g2}, \ref{LemmaUrysohnNotExtend} and Theorem \ref{UrysohnMainTheorem} follow easily for $\mathbb{S}$ if we make the obvious changes. There is also a final remark that should be made. The theory of \Fraisse{} limits grafted with ideas from continuous logic can be naturally generalized to the context of complete separable metric structures; see for example \cite{Itai}. In this context, the Urysohn space and the Urysohn sphere are just examples of the general theory. A natural question arises. Namely, whether a dichotomy similar to the one in Chapter \ref{ChapterSAP} could possibly hold for ``metric \Fraisse{} structures''. The following example suggests which types of metric structures would belong to the ``global" side of the dichotomy. However, the methods developed here face some obstacles when we try to apply them into this context. For example, the general theory of ``metric \Fraisse{} structures'' is developed for \emph{approximately ultrahomogeneous} structures rather than ultrahomogeneous structures and moreover, a natural notion of SAP does not seem to exist.
\begin{example}\label{Baire}
Let $\mathcal{N}=\N^\N$ be the Baire space endowed with the ultrametric $d$ with 
\[d(\alpha,\beta)=\frac{1}{m}\quad\text{where}\quad m=\min\{n: \alpha(n)\neq\beta(n)\},\]
if $\alpha\neq\beta$ and $d(\alpha,\beta)=0$ otherwise. The metric structure $(\mathcal{N},d)$ is a metric \Fraisse{} structure that happens to be ultrahomogeneous. For a generic subspace $F\in\mathcal{F}(\mathcal{N})$, $F$ is the body of a pruned tree $T$ on $\N$ such that for every  $n\in\N$ there are infinitely many $s\in T$ and infinitely many $s\not\in T$ of length $n$. It is easy now to see that for a generic $F\in\mathcal{F}(\mathcal{N})$, $F$ is a global substructure and that the generic pair $A,B\in\mathcal{F}(\mathcal{N})$ is also global.
\end{example}


\end{document}